\documentclass[12pt,twoside,reqno]{amsart}

\usepackage[colorlinks=true,citecolor=blue]{hyperref}
\usepackage{mathptmx,amsmath,amssymb,amsfonts,amsthm,mathptmx,enumerate,color,mathrsfs}

\usepackage{graphicx}
\usepackage{multirow}
\usepackage{epstopdf}
\usepackage{multicol}
\usepackage{algorithm}
\usepackage{algorithmic}
\usepackage{epstopdf}

\newtheorem{theorem}{Theorem}[section]
\newtheorem{lemma}[theorem]{Lemma}

\theoremstyle{definition}
\newtheorem{definition}[theorem]{Definition}

\newtheorem{remark}[theorem]{Remark}
\numberwithin{equation}{section}

\DeclareMathOperator{\mes}{meas}
\DeclareMathOperator{\sgn}{sgn}

\begin{document}

\setcounter{page}{1}


\vspace*{2.0cm}
\title[Near-optimal control of a stochastic SICA model with imprecise parameters]{%
Near-optimal control of a stochastic SICA model\\ with imprecise parameters}
\author[H. Zine, D. F. M. Torres]{Houssine Zine, Delfim F. M. Torres$^{*}$}
\maketitle
\vspace*{-0.6cm}


\begin{center}
{\footnotesize
Center for Research and Development in Mathematics and Applications (CIDMA),\\
Department of Mathematics, University of Aveiro, 3810-193 Aveiro, Portugal\\[0.3cm]

(In memory of Professor Jack Warga on the occasion of his 100th birthday)}
\end{center}


\vskip 4mm {\footnotesize \noindent {\bf Abstract.}
An adequate near-optimal control problem for a stochastic 
SICA (Susceptible--Infected--Chronic--AIDS)
compartmental epidemic model for HIV transmission
with imprecise parameters is formulated and investigated. 
We prove some estimates for the state and co-state variables
of the stochastic system. The established inequalities are 
then used to prove a necessary and a sufficient condition
for near-optimal control with imprecise parameters. 
The proofs involve several mathematical and stochastic tools, 
including the Burkholder--Davis--Gundy inequality.

}


\renewcommand{\thefootnote}{}
\footnotetext{ $^*$Corresponding author.
\par
E-mail address: zinehoussine@ua.pt (H. Zine), delfim@ua.pt (D. F. M. Torres).
\par
Received June, 27, 2022; Revised and Accepted September, 27, 2022.
	
\rightline{This is a preprint of a paper whose final and definite 
form is published in \emph{Commun. Optim. Theory}}}


\section{Introduction}

Despite the advance of medical knowledge and technology, 
infectious diseases remain a growing threat to mankind.
Indeed, various kinds of infectious diseases, including hepatitis C, 
HIV/AIDS and COVID-19, spread around the globe.
Therefore, decision-makers and public health systems build up different 
strategies to curb the spread of diseases. Mathematical modeling is 
nowadays an important tool in analyzing the growth and in controlling 
pandemics \cite{MyID:461,MyID:468}.

The infection by human immunodeficiency virus (HIV) is still a major public issue, 
where no cure or vaccine exists for the acquired immunodeficiency syndrome (AIDS). 
However, antiretroviral (ART) treatment improves health, prolongs life, and diminishes 
the risk of HIV infection. Several mathematical models have been proposed 
for HIV/AIDS transmission dynamics with ART control measures, see, e.g., 
\cite{MR4376318,MR4175342} and references therein. Here we are mainly motivated 
by \cite{R1}, where a stochastic SICA (Susceptible--Infected--Chronic--AIDS)
compartmental epidemic model for HIV transmission is proposed and investigated; 
and by \cite{R10}, where SICA modeling is shown to be 
a very useful tool with respect to real applications of HIV/AIDS.

The SICA model was introduced by Silva and Torres in 2015, 
as a sub-model of a general Tuberculosis and HIV/AIDS 
co-infection problem \cite{MR3392642}. After that, 
it has been extensively used to investigate HIV/AIDS,
in different settings and contexts, using fractional-order 
derivatives \cite{MR3980195}, stochasticity \cite{R1} 
and discrete-time operators \cite{MR4276109}, 
and adjusted to different HIV/AIDS epidemics, 
as those in Cape Verde \cite{R10} 
and Morocco \cite{MR3999700}. 

In \cite{R1}, Djordjevic, Silva and Torres propose a stochastic SICA model 
by means of white noise (Brownian motion with positive intensity) 
due to environment fluctuations that perturb the coefficient rate 
of transmission $\beta$ into $\beta+\delta B_t$ \cite{R5,R7,R12}.
The stochastic SICA model of Djordjevic et al. has no control 
and is given as follows \cite{R1}:
\begin{equation}
\label{1}
\left\lbrace
\begin{aligned}
dS(t) &=\left[ \Lambda -\beta \left( I(t) +\eta _{C} C(t) 
+\eta_{A} A(t) \right) S(t) -\mu S(t) \right] dt\\
&\quad -\delta \left( I(t) +\eta_{C} C(t) 
+\eta _{A} A(t) \right) S(t) dB(t),\\
dI(t) &=\left[ \beta \left( I(t) +\eta_{C} C\left(t\right) 
+\eta _{A} A(t) \right) S(t) 
-\varepsilon_{3}I(t) +\alpha A(t) +\omega C(t)\right] dt\\
&\quad +\delta \left( I(t) +\eta _{C} C(t) 
+\eta _{A} A(t) \right) S(t) dB(t),\\
dC(t)& =\left[ \phi I(t) -\varepsilon _{2}C\left( t\right) \right] dt,\\
dA(t) &=\left[ eI(t) -\varepsilon _{1}A\left( t\right) \right] dt,
\end{aligned}
\right.
\end{equation}
with $\varepsilon_{1}=\alpha +\mu +d$,
$\varepsilon_{2}=\omega +\mu$
and $\varepsilon_{3}=e+\phi +\mu$,
and where the meaning of the parameters
$\Lambda$, $\beta$, $\mu$, 
$\eta_A$, $\eta_C$, $\phi$, $e$, $\alpha$, 
$\omega$, $C$ and $d$ is given in Table~\ref{tab1}.
\begin{table}[h!]
\label{tab1}
\caption{Parameters of the stochastic 
SICA model \eqref{1} and their meaning.}
\centering
\begin{tabular}{|c | c |} \hline
Parameters &    Meaning  \\ \hline \hline
$\Lambda$ &  Recruitment rate\\ \hline
$\beta$ & Transmission rate\\ \hline
$\mu$  &  Natural death rate \\ \hline
$\eta_A$ &  Relative infectiousness of
individuals with AIDS symptoms\\ \hline
$\eta_C$   & Partial restoration of immune function of individuals \\ 
           &    with HIV infection that use ART correctly\\ \hline
$\phi$ & HIV  treatment rate for $I$ individuals\\ \hline
$e$ & Default treatment rate for $I$ individuals\\ \hline
$\alpha$ & AIDS treatment rate\\ \hline
$\omega$ & Default treatment rate for $C$ individuals\\ \hline
$d$ & AIDS induced death rate\\ \hline
\end{tabular}
\end{table}

Thus, in SICA modeling, the human population is subdivided 
into four exhaustively and mutually exclusive compartments: 
susceptible individuals ($S$); HIV-infected individuals ($I$) with no clinical symptoms of AIDS 
(the virus is living or developing in the individuals but without producing symptoms or only mild ones) 
but able to transmit HIV to other individuals; HIV-infected individuals under ART treatment 
(the so called chronic stage) with a viral load remaining low ($C$); and HIV-infected  
individuals with AIDS clinical symptoms ($A$). The total population at time $t$, 
denoted by $N(t)$, is given by $N(t)=S(t)+I(t)+C(t)+A(t)$.

In Section~\ref{sec2}, we apply to the SICA model \eqref{1}
a meaningful control $u$. The Hamiltonian function is then introduced 
and we end up by recalling the Pontryagin maximum principle. 
Motivated by \cite{R13,R9}, our main results 
are then given in Section~\ref{sec3},
where we replace known biological parameters by imprecise ones, 
taking into account all possible environment perturbations.
Near-optimal control is considered and,
with the help of different mathematical techniques, 
like the inequalities of Cauchy--Schwartz, H\"{o}lder, 
and Burkholder--Davis--Gundy, It\^{o}'s formula, 
convexity, and Ekeland's variational principle, 
we prove estimates for the state (Lemmas~\ref{3.2} and \ref{3.3})
and co-state variables of the system (Lemmas~\ref{3.4} and \ref{3.5}),
a necessary (Theorem~\ref{Th2}) and a sufficient condition 
for near-optimal control (Theorem~\ref{Th3}). 
We end with Section~\ref{sec4} of conclusion.


\section{Optimal control of the stochastic SICA model}
\label{sec2}

To obtain an adequate stochastic controlled SICA model, 
we introduce a control into \eqref{1}
using the procedure presented in \cite{R9} 
for the stochastic SIRS model. In concrete, we propose 
the following dynamical control system:
\begin{equation}
\label{2}
\left\lbrace
\begin{aligned}
dS(t) &=\left[ \Lambda -\beta \left( I(t) +\eta_{C} C(t) 
+\eta _{A} A(t) \right) S(t)
-\mu S(t) \right] dt\\
&\quad -\delta \left( I(t) +\eta_{C} C(t) 
+\eta_{A} A(t) \right) S(t) dB(t), \\
dI(t) &=\left[ \beta \left( I(t) +\eta _{C} C\left(t\right) 
+\eta_{A} A(t) \right) S(t) -\varepsilon_{3}I(t) 
+\alpha A(t) +\omega C(t) 
-\dfrac{mu(t) I(t) }{1+\gamma I(t) }\right] dt \\
&\quad +\delta \left( I(t) +\eta _{C} C(t) +\eta_{A}
A(t) \right) S(t) dB(t),\\
dC(t) &=\left[ \phi I(t) -\varepsilon _{2}C\left(t\right) 
+\dfrac{mu(t) I(t) }{1+\gamma I\left(
t\right) }\right] dt, \\
dA(t) &=\left[ eI(t) -\varepsilon _{1}A\left(
t\right) \right] dt.
\end{aligned}
\right.
\end{equation}
Here $m$ represents the adjustment coefficient of the control. 
To simplify the writing and the analysis, we define
the vector $x = (x_1, x_2, x_3, x_4)$ as
$$
x(t) :=\left( S(t) ,I(t) ,C\left(t\right),A(t) \right)
$$
and
\begin{equation}
\label{5}
\begin{split}
f_{1}\left( x \right) &=\Lambda -\beta \left( x_{2}+\eta_{C} x_{3}
+\eta_{A} x_{4}\right) x_{1}-\mu x_{1}, \\
f_{2}\left(x,u\right)  &=\beta \left(x_{2}+\eta_{C} x_{3}
+\eta_{A} x_{4}\right) x_{1}-\varepsilon_{3}x_{2}
+\alpha x_{4}+\omega x_{3}-\dfrac{m u x_{2}}{1+\gamma x_{2}},\\
f_{3}\left( x, u \right)  &=\phi x_{2}-\varepsilon _{2}x_{3}
+\dfrac{m u x_{2}}{1+\gamma x_{2}},\\
f_{4}\left( x \right) &=ex_{2}-\varepsilon _{1}x_{4,} \\
\sigma _{1}\left( x \right)  &=-\sigma _{2}\left( x \right)
=-\delta \left( x_{2}+\eta_{C} x_{3}+\eta _{A} x_{4}\right) x_{1}.
\end{split}
\end{equation}
With these notations, we write our SICA control system \eqref{2} as
\begin{equation}
\label{4}
\left\lbrace
\begin{aligned}
dx_{1}(t) &= f_{1}\left( x(t) \right) dt+\sigma _{1}\left( x(t) \right) dB_{t}, \\
dx_{2}(t) &= f_{2}\left( x(t),u(t) \right) dt+\sigma _{2}\left( x(t) \right) dB_{t}, \\
dx_{3}(t) &= f_{3}\left( x(t),u(t) \right) dt,  \\
dx_{4}(t) &= f_{4}\left( x(t)\right) dt.
\end{aligned}
\right.
\end{equation}

Now, we introduce the stochastic objective/cost functional
$J\left( u\right)$ as
\begin{equation}
\label{6}
J\left( u\right) =E\left\{ \int_{0}^{T}L\left( x(t) ,u\left(
t\right) \right) dt+h\left( x(T) \right) \right\}
\end{equation}
with $L:\mathbb{R}^{4}\times \mathbb{R}\longrightarrow \mathbb{R}$
and $h:\mathbb{R}^{4}\rightarrow\mathbb{R}$ 
assumed to be continuously differentiable.

Let $U\subseteq \mathbb{R}$ be a given bounded nonempty closed set.
A control $u :\left[ 0,T\right] \longrightarrow U$ 
is called admissible if it is an $F_{t}$-adapted process 
with values in $U$. The set of all admissible controls is denoted by $U_{ad}$.

Our optimal control problem consists to find an admissible control that
minimizes the objective functional $J\left( u\right)$ subject to the control
system \eqref{4} and a given initial condition $x(0) = x_{0}$.

Associated to our stochastic optimal control problem, 
we define the stochastic Hamiltonian function by
\begin{equation}
\label{9}
H\left(x,u,p,q\right) =\left\langle f\left( x,u\right) ,p\right\rangle
+\left\langle \sigma ,q\right\rangle -L\left( x,u\right),
\end{equation}
where $\sigma =\left( \sigma _{1},\sigma _{2}\right)$
and $\left\langle \cdot , \cdot \right\rangle$
denotes the Euclidean inner product.

The stochastic Pontryagin's Maximum Principle \cite{MR0186436,R7}
asserts that if $u^{\ast}$ is an optimal control and
$x^{\ast}$ is the corresponding optimal trajectory, then 
there exists nontrivial multipliers $p$ and $q$ such that
the Hamiltonian system 
\begin{equation}
\label{10:a}
\left\lbrace
\begin{aligned}
dx^{\ast }(t) &=\frac{\partial H}{\partial p}\left( x^{\ast}(t),
u^{\ast }(t),p(t),q(t)\right) dt+\sigma\left( x^{\ast }(t) \right)dB_{t},\\
dp_{i}(t) &=-\frac{\partial H}{\partial x_i}\left(x^{\ast}(t),u^{\ast }(t),
p(t),q(t)\right) dt +q_{i}(t) dB_{t}, \quad i = 1, 2, 3, 4,  
\end{aligned}
\right.
\end{equation}
holds together with the maximality condition
\begin{equation}
\label{10:b}
H\left( x^{\ast }(t),u^{\ast }(t),p(t),q(t)\right) 
=\max\limits_{v\in U}H\left(x^{\ast}(t),v,p(t),q(t)\right)
\end{equation}
and the initial state and terminal costate conditions
\begin{equation}
\label{11}
x^{\ast }\left( 0\right) =x_{0},
\quad p_{i}(T) =-\frac{\partial h}{\partial x_{i}}\left(x^{\ast }(T) \right).
\end{equation}
Note that in our problem the diffusion term $\sigma$ does not depend
on the control $u$.

In Section~\ref{sec3}, we transform \eqref{2}
into a near-optimal control problem 
via the incorporation of imprecise parameters.


\section{Near-optimal control with imprecise parameters}
\label{sec3}

In the majority of available mathematical epidemic models, 
the parameter values are assumed to be precisely known. 
Nevertheless, in real applications, one needs to take into account 
the influence of numerous uncertainties. This motivate us to 
include here, for the first time in the literature of SICA modeling,
imprecise parameters into the stochastic SICA model 
and to consider the problem of near-optimal control. 
Before that, we need some preliminary notions.


\subsection{Preliminaries}
\label{sec3.1}

To make it explicitly that in our optimal control problem
we begin at time $t = 0$ with the given initial state $x(0) = x_0$,
from now on we denote the cost functional $J\left( u\right)$
defined in \eqref{6} by $J(0,x_0,u)$. Moreover,
the minimum of functional \eqref{6} is denoted by
$V(0, x_0)$, that is, 
$$
V(0, x_0) = \min_{u \in U_{ad}} J(0,x_0,u).
$$
Function $V$ is known in the literature as the value function.

We recall some known definitions available
in the literature: see, e.g., \cite{R2,R9}. 

\begin{definition}[optimal control]
\label{(D1)}
An admissible control $u^*$ is called optimal if 
$$
J(0,x_0,u^*) = V(0, x_0).
$$
\end{definition}

\begin{definition}[$\epsilon$-optimal control]
\label{(D2)}
Let $\epsilon>0$. An admissible control $u^\epsilon$ 
is called  $\epsilon$-optimal if
$$
\mid J(0,x_0,u^\epsilon)-V(0,x_0)\mid\leq\epsilon.
$$
\end{definition}

\begin{definition}[near-optimal control]
\label{(D3)}
Consider a family of admissible controls $\{u^\epsilon\}$ parameterized 
by $\epsilon >0$ and let $u^\epsilon$ be any element in this family. 
We say that $u^\epsilon$ is a near-optimal control if
$$
\mid J(0,x_0,u^\epsilon)-V(0,x_0)\mid\leq \delta(\epsilon)
$$
holds for a sufficient small $\epsilon >0$, 
where $\delta$ is a function of $\epsilon$ satisfying 
$\delta(\epsilon)\rightarrow0$ as $\epsilon\rightarrow 0$.
The estimate $\delta(\epsilon)$ is called an error bound.
If $\delta(\epsilon)= r\epsilon^\omega$, for some $r,\omega>0$, 
then $u^\epsilon$ is said to be a near-optimal control 
of order $\epsilon^\omega$.
\end{definition}

\begin{definition}[interval numbers]
\label{(D4)}
An interval number $A$ is represented by a closed interval $[a_l,a_u]$ defined by
$$
A=[a_l,a_u]=\left\{x\in\mathbb{R} : a_l\leq x \leq a_u \right\},
$$
where $a_l$ and $a_u$ are the lower and the upper limits of the interval number, 
respectively. We represent an interval $[a,b]$ by the so called 
interval-valued function, which is given by
$$
h(k)=a^{1-k}b^k, \quad k\in [0,1].
$$
\end{definition}

\begin{remark}
The sum, difference, product and division of two interval numbers 
are also interval numbers. 
\end{remark}


\subsection{The stochastic SICA control model with imprecise parameters}
\label{sec3.2}

We assume that the stochastic SICA model has some biological imprecise parameters.
The uncertain parameters are described by interval numbers. 
After replacing each parameter $\zeta$ with an imprecise one, 
$\zeta_k=(\zeta_l)^{1-k}(\zeta_u)^k \in [\zeta_l,\zeta_u]$ for $k\in [0,1]$,
our control system \eqref{2} becomes:
\begin{equation}
\label{13}
\left\lbrace
\begin{aligned}
dS(t) &=\left[ \Lambda_k -\beta_k \left( I(t) +(\eta_{C})_k C(t) 
+(\eta _{A})_k A(t) \right) S(t) -\mu_k S(t) \right] dt\\
&\quad -\delta_k \left( I(t) +(\eta_{C})_k C(t) +(\eta _{A})_k A(t) \right) S(t) dB(t), \\
dI(t) &=\Biggl[ \beta _k\left( I(t) +(\eta _{C})_k C\left(t\right)
+(\eta _{A})_k A(t) \right) S(t) -(\varepsilon_{3})_k I(t)\\
&\quad +\alpha_k A(t) +\omega_k C(t)-\dfrac{m_k u(t) I(t)}{1+\gamma_k I(t)}\Biggr] dt \\
& \quad +\delta_k \left( I(t) +(\eta _{C})_k C(t) +(\eta_{A})_k A(t) \right) S(t) dB(t),\\
dC(t) &=\left[ \phi_k I(t) -(\varepsilon _{2})_k C\left(t\right) 
+\dfrac{m_ku(t) I(t) }{1+\gamma_k I\left(t\right) }\right] dt, \\
dA(t) &=\left[ e_k I(t) -(\varepsilon_{1})_k A\left(t\right) \right] dt.
\end{aligned}
\right.
\end{equation}

Let $u$ and $u'\in U_{ad}$. Set the following metric on $U_{ad}[0,T]$: 
\begin{equation}
\label{eq:metric:d}
d(u,u')=E\mes \{{t\in [0,T]:u(t) \neq u'(t)}\},
\end{equation}
where ``$\mes$'' represents the Lebesgue measure.
Note that since $U$ is closed, it follows that $U_{ad}$ 
is a complete metric space under $d$.

Next we prove estimates of the state and co-state variables. Let 
\begin{equation}
\label{eq:Omega}
\Omega=\left\{(S(t),I(t), C(t),A(t))\in \mathbb{R}_+^4 :
\dfrac{\Lambda_k}{\mu_k+d_k}\leq N(t)
\leq  \dfrac{\Lambda_k}{\mu_k}\right\}\subset \mathbb{R}_+^4,
\end{equation}
where $N(t)=S(t)+I(t)+C(t)+A(t)$. Shortly, our
Lemma~\ref{3.2} asserts that the trajectories of \eqref{13} 
will enter and remain in $\Omega$ with probability 1.

\begin{lemma}
\label{3.2}
For any $\theta\geq 0$ and $u\in U_{ad}$, we have
$$
E\underset{0\leq t\leq r}{\sup}(\mid S(t)\mid^\theta
+\mid I(t)\mid^\theta+\mid C(t)\mid^\theta+\mid A(t)\mid^\theta)\leq R,
$$
where $R$ is an imprecise parameter depending only on $\theta$.
\end{lemma}

\begin{proof}
Adding member to member all equations 
of the system \eqref{13}, we get
$$
\dfrac{dN(t)}{dt}=\Lambda_k-\mu_k N(t)-d_k A(t).
$$
Because $d_k A(t)\geq 0$, it follows that
$$
\dfrac{dN(t)}{dt}\leq \Lambda_k-\mu_k N(t).
$$
Multiplying by $\exp(u_k t)$ both sides of this inequality, we obtain that
$$
\exp(u_k t) \, \dfrac{dN(t)}{dt} \leq \Lambda_k \exp(u_k t)
-\mu_k \exp(u_k t) N(t).
$$
An integration by parts between $0$ and $t$ leads to
$$
N(t)\exp(u_k t)-N(0)-\int_0^t N(s)\exp(u_k s)ds 
\leq \dfrac{\Lambda_k}{\mu_k}(exp(u_k t)-1)
-\int_0^t \mu_k \exp(u_k s) N(s)ds.
$$
Equivalently, we have
$$
N(t) \leq \dfrac{\Lambda_k}{\mu_k}(1-\exp(-u_k t))+N(0)\exp(-u_k t)
$$
and thus
$$
\underset{t\rightarrow+\infty}{\limsup} \, N(t)\leq \dfrac{\Lambda_k}{\mu_k}.
$$
We also have
$$
\dfrac{dN(t)}{dt}\geq \Lambda_k-(\mu_k+d_k) N(t)
$$
and, following the same arguments as before, we get
$$
\underset{t\rightarrow+\infty}{\liminf}\, N(t)\geq \dfrac{\Lambda_k}{\mu_k+d_k}.
$$
We conclude that 
$$
\dfrac{\Lambda_k}{\mu_k+d_k}\leq \underset{t\rightarrow+\infty}{\liminf}\, N(t)
\leq \underset{t\rightarrow+\infty}{\limsup}\, N(t)\leq  \dfrac{\Lambda_k}{\mu_k}.
$$
This means that all solutions $S(t)$, $I(t)$, $C(t)$, and $A(t)$ of our model \eqref{13} 
are almost surely bounded over the positively invariant bounded set \eqref{eq:Omega}.
\end{proof}

In what follows we use the notation
$x(t)=(x_1(t),x_2(t), x_3(t), x_4(t))=(S(t),I(t), C(t), A(t))$
for our system \eqref{13}.

\begin{lemma}
\label{3.3}
Let $\theta \geq 0$ and $0<k<1$ with $k\theta<1$. 
If $u, u'\in U_{ad}$ and $x$ and $x'$
are the corresponding state trajectories, 
then there exists an imprecise parameter 
$R=R(\theta,k)$ such that
\begin{equation}
\label{eq:Lema3.7}
\sum_{i=1}^4E\underset{0\leq t\leq T}{\sup}\mid x_i(t)
-x'_i(t)\mid ^{2\theta}\leq R d(u(t),u'(t))^{k\theta}.
\end{equation}
\end{lemma}

\begin{proof}
Let us first suppose that $\theta \geq 1$. 
From system \eqref{13}, and in view of 
H\"{o}lder's inequality, for $\theta\geq 1$ we have
\begin{equation}
\label{ineq:01}
\begin{aligned}
E \underset{0\leq t \leq r}{\sup}\mid x_1(t)-x'_1(t)\mid^{2\theta}
&\leq R E\int_0^r[(\beta_k^{2\theta}+\delta_k^{2\theta})\mid 
(x_2+(\eta_C)_k x_3+(\eta_A)_k x_4) x_4\\
&\quad -((x'_2+(\eta_C)_k x'_3+(\eta_A)_k x{'}_4)) x_4\mid^{2\theta}
+(\mu_k)^{2\theta}\mid x_1(t)-x_1{'}(t)\mid^{2\theta}]dt\\
&\leq R E\int_0^r \sum_{i=1}^4\mid x_i-x{'}_i\mid^{2\theta} dt.
\end{aligned}
\end{equation}
Since
\begin{equation*}
\begin{aligned}
E\int_0^r \chi_{u(t)\neq u'(t)}dt
&\leq  \left(E \int_0^r dt\right)^{1-k\theta}\left( 
E\int_0^r \chi_{u(t)\neq u'(t)}dt\right)^{k\theta}\\
&\leq R(\theta,k)\left( E \mes\{t/ u(t)\neq u'(t)\}\right)^{k \theta}\\
&\leq R d(u,u')^{k\theta},
\end{aligned}
\end{equation*}
we also have
\begin{equation}
\label{ineq:02}
\begin{aligned}
E \underset{0\leq t \leq r}{\sup}\mid x_2(t)-x'_2(t)\mid^{2\theta}
&\leq R E\int_0^r\Biggl[(\beta_k^{2\theta}+\delta^{2\theta})
\mid (x_2+(\eta_C)_k x_3+(\eta_A)_k x_4) x_4\\
&\quad - ((x'_2+(\eta_C)_k x'_3+(\eta_A)_k x{'}_4)) x_4 \mid^{2\theta}\\
&\quad +(\epsilon_3)_k^{2\theta}\mid x_2(t)-x{'}_2(t)
\mid^{2\theta}+\alpha_k^{2\theta}\mid x_4(t)-x'_4(t)\mid^{2\theta}\\
&\quad  + \omega_k^{2\theta} \mid x_3(t)-x{'}_3(t)\mid^{2\theta}\\
&\quad  + m_k^{2\theta}\left| \dfrac{u(t) x_2(t)}{1+\eta_k x_2}
-\dfrac{u'(t) x_2'(t)}{1+\eta_k x'_2}\right|^{2\theta}\Biggr]dt\\
&\leq R E\int_0^r \sum_{i=1}^4\mid x_i-x{'}_i\mid^{2\theta}
+ R E\int_0^r \chi_{u(t)\neq u'(t)}dt\\
&\leq R\left[E\int_0^r \sum_{i=1}^4\mid 
x_i-x{'}_i\mid^{2\theta}dt+d(u,u')^{k\theta}\right].
\end{aligned}
\end{equation}
From H\"{o}lder's inequality under the hypothesis $k \theta <1$,
we also find that
\begin{equation}
\label{ineq:03}
E \underset{0\leq t \leq r}{\sup}\mid x_3(t)-x'_3(t)\mid^{2\theta}
\leq R\left[E\int_0^r \sum_{i=1}^4\mid x_i
-x{'}_i\mid^{2\theta}dt+d(u,u')^{k\theta}\right]
\end{equation}
and
\begin{equation}
\label{ineq:04}
E \underset{0\leq t \leq r}{\sup}\mid x_4(t)
-x'_4(t)\mid^{2\theta}\leq  
R\left[E\int_0^r \sum_{i=1}^4\mid x_i-x{'}_i\mid^{2\theta}dt
+d(u,u')^{k\theta}\right].
\end{equation}
Combining the last four inequalities 
\eqref{ineq:01}--\eqref{ineq:04}, we get
$$
\sum_{i=1}^4 E \underset{0\leq t \leq r}{\sup}\mid x_i(t)
-x'_i(t)\mid^{2\theta}\leq  R\left[E\int_0^r 
\sum_{i=1}^4\mid x_i-x{'}_i\mid^{2\theta}dt+d(u,u')^{k\theta}\right].
$$
By using Gronwall's inequality, we conclude that \eqref{eq:Lema3.7} holds. 

To prove the desired result in the case $0\leq \theta < 1$, we apply 
the Cauchy--Schwartz inequality for obtaining
\begin{equation*}
\begin{aligned}
\sum_{i=1}^4 E \underset{0\leq t \leq r}{\sup}\mid x_i(t)-x'_i(t)\mid^{2\theta}
& \leq \sum_{i=1}^4\left[E \underset{0\leq t \leq r}{\sup}\mid 
x_i(t)-x'_i(t)\mid^2 \right] ^\theta\\
& \leq \left[ R d(u,u')^k\right] ^\theta\\
& \leq R^\theta d(u,u')^{k\theta},
\end{aligned}
\end{equation*}
where $R$ is the imprecise parameter. The proof is complete.
\end{proof}

Now we prove estimates for the co-state variables.

\begin{lemma}
\label{3.4}
Let $p_i$ be the co-state variables given by the stochastic 
Pontryagin's maximum principle. Then, 
\begin{equation*}
\sum_{i=1}^4 E \underset{0\leq t \leq T}{\sup}\mid 
p_i(t)\mid^2+\sum_{i=1}^2 E\int_0^T\mid q_i(t)\mid^2dt \leq R,
\end{equation*}
where $R$ is an imprecise parameter.
\end{lemma}

\begin{proof}
From Pontryagin's maximum principle we have
$$
dp_1(t)=-\partial_{x_{1}}H(x^*,u^*,p,q)dt+q_1(t)dB_t=-b_1dt+q_1dB_t,
$$
so that
$$
p_1(t)+\int_t^T q_1(s)dB_s=p_1(T)+\int_t^T b_1(s)ds.
$$
For all $t\geq 0$, and since by Lemma~\ref{3.2} $x^*(t)\in \Omega$,
one has
\begin{equation*}
\begin{aligned}
E\mid p_1(t)\mid^2+E\left\{\int_0^T\mid q_1(s)\mid^2ds\right\}
\leq R E\mid p_1(T)\mid^2+ R(T-t)E\left\{\int_0^T\mid b_1(s)\mid^2ds\right\}\\
\leq R E\mid p_1(T)\mid^2+ R (T-t)\sum_{i=1}^4E\left\{\int_0^T\mid 
p_i(s)\mid^2ds\right\}+ R (T-t)\sum_{i=1}^2 E\left\{\int_0^T\mid q_i(s)\mid^2ds\right\}.
\end{aligned}
\end{equation*}
Similarly, we obtain the same inequalities for $i\in\{2,3,4\}$.
By adding member to member, we get
\begin{equation*}
\begin{aligned}
&\sum_{i=1}^4E\mid p_i(t)\mid^2+\sum_{i=1}^2E\int_t^T\mid q_i(s)\mid^2ds\\
&\leq R E\mid p_i(T)\mid^2+ R (T-t)\sum_{i=1}^4E
\left\{\int_t^T\mid p_i(s)\mid^2ds\right\}
+ R (T-t)\sum_{i=1}^2E\left\{\int_t^T\mid q_i(s)\mid^2ds\right\}.
\end{aligned}
\end{equation*}
Then,
\begin{equation*}
\begin{aligned}
&\sum_{i=1}^4E\mid p_i(t)\mid^2+\frac{1}{2}\sum_{i=1}^2
E\left\{\int_t^T\mid q_i(s)\mid^2ds\right\}\\
&\qquad \leq \sum_{i=1}^4 R E\mid p_i(T)\mid^2+ R(T-t)\sum_{i=1}^4
E\left\{\int_t^T\mid p_i(s)\mid^2ds\right\},
\end{aligned}
\end{equation*}
where $t\in[T-\epsilon,T]$ and $\epsilon=\frac{1}{R}$.
Using Gronwall's inequality, we obtain that
\begin{equation}
\sum_{i=1}^4 \underset{0\leq t \leq T}{\sup}E\mid p_i(t)\mid^2\leq R,
\quad \sum_{i=1}^4 E\left\{\int_{t}^T\mid q_i(s)\mid^2ds\right\}\leq R.
\end{equation}
We also obtain the same result over $[T-2\epsilon,T]$,  $[T-3\epsilon,T]$, and so on.
Repeating for a finite number of steps, 
the expected estimate emerges for any $t\in [0,T]$. Furthermore, from 
$$
p_1(t)=p_1(T)+\int_t^Tb_1(s)ds-\int_0^Tq_1(s)dB_s+\int_0^tq_1(s)dB_s
$$
and the elementary inequality 
\begin{equation}
\label{el:ineq}
\mid m_1+m_2+m_3+m_4\mid^n
\leq 4^n\left( \mid m_1\mid^n+ \mid m_2\mid^n
+ \mid m_3\mid^n+ \mid m_4\mid^n\right),
\end{equation}
valid for any $n>0$, we have 
\begin{multline*}
\mid p_1(t)\mid^2
\leq R \Biggl[\mid p_1(T)\mid^2
+\int_0^T\left(\sum_{i=1}^4\mid p_i(s)\mid^2\right)ds
+\int_0^T\left(\sum_{i=1}^2\mid q_i(s)\mid^2\right)ds\\
+\left(\int_0^Tq_1(s)dB\right)^2
+\left(\int_0^tq_1(s)dB\right)^2\Biggr].
\end{multline*}
Analogous statements are established for $p_2$, $p_3$, and $p_4$.
Next, by addition, we get
\begin{multline*}
\sum_{i=1}^4\mid p_i(t)\mid^2
\leq R \Biggl[\sum_{i=1}^4\mid p_i(T)\mid^2
+\int_0^T\left(\sum_{i=1}^4\mid p_i(s)\mid^2\right)ds\\
+\int_0^T\left(\sum_{i=1}^2\mid q_i(s)\mid^2\right)ds
+\sum_{i=1}^2\left(\int_0^Tq_i(s)dB\right)^2
+\sum_{i=1}^2\left(\int_0^tq_i(s)dB\right)^2\Biggr].
\end{multline*}
The Burkholder--Davis--Gundy inequality \cite{R5} leads to
\begin{equation*}
\begin{aligned}
\sum_{i=1}^4E\underset{0\leq t \leq T}{\sup}\mid p_i(t)\mid^2
&\leq R\Biggl[\sum_{i=1}^4\mid p_i(T)\mid^2
+\sum_{i=1}^4E\left\{\int_0^T\underset{0\leq v \leq s}{\sup}\mid 
p_i(v)\mid^2ds\right\}\\
&\quad +\sum_{i=1}^2E\left\{
\int_0^T\underset{0\leq v \leq s}{\sup}\mid q_i(v)\mid^2 ds\right\}\Biggr]
\end{aligned}
\end{equation*}
and the establishment of our desired result 
is achieved by applying Gronwall's inequality.
\end{proof}

The following assumptions will be used in our next results:
\begin{enumerate}
\item[$(S_1)$] For all $0\leq t\leq T$, the partial derivatives 
$L_{S}$, $L_{I}$, $L_C$, $L_{A}$ and 
$h_{S}$ and $h_{I}$ are continuous, 
and there exists an imprecise parameter $R$ such that
$$ 
\mid L_{S} + L_{I} + L_{C}+L_{A}\mid 
\leq R (1+\mid S(t)\mid+\mid I(t)\mid+\mid C(t)\mid+\mid A(t)\mid).
$$
Moreover, 
$$
(1+\mid S(t)\mid )^{-1}\mid h_{S}\mid 
+ (1+\mid I(t)\mid )^{-1}\mid h_{I}\mid\leq R.
$$
	
\item[$(S_2)$] If $x(t), x' (t) \in \mathbb{R}_+^4$
for any $0\leq t\leq T$,  $u, u' \in U_{ad}$,
and function $u \mapsto L(x,u)$ is differentiable, 
then there exists an imprecise parameter $R$ such that
$$
\mid L(x(t),u(t))-L(x(t),u'(t))\mid+\mid L_{u}(x(t),u(t))
-L_{u}(x(t),u'(t))\mid\leq R \mid u(t)-u'(t)\mid.
$$
Moreover, if $h$ is differentiable, then
$$ 
\sum_{i=0}^4\mid h_{x_i}(x(t))- h_{x_i}(x'(t))\mid 
\leq R \sum_{i=0}^4\mid x_i(t)-x_i'(t)\mid.
$$
\end{enumerate}

\begin{lemma}
\label{3.5}
Let $(S_1)$ and $(S_2)$ hold. For any $1<\eta<2$ 
and $0<k<1$ satisfying $(1+k)\eta<2$, 
there exists a constant $R=R(\eta,k)$ such that 
for any $u,u'\in U_{ad}$ along with the corresponding 
trajectories $x,x'$ and adjoint multipliers $(p,q),(p',q')$, we have
\begin{equation*}
\sum_{i=1}^4E\left\{\int_0^T \mid p_i(t)-p'_i(t)\mid^\eta dt\right\}
+\sum_{i=1}^4E\left\{\int_0^T \mid q_i(t)-q'_i(t)\mid^\eta dt\right\}
\leq R d(u,u')^\frac{k\eta}{2}.
\end{equation*}
\end{lemma}

\begin{proof}
Let $\bar{p}_i(t)=p_i(t)-p'_i(t)$, $i\in\{1,2,3,4\}$, and  
$\bar{q}_i(t)=q_i(t)-q'_i(t)$, $i\in\{1,2\}$.
It follows from \eqref{10:a} that
\begin{equation*}
\begin{cases}
d\bar{p}_1(t)
= -[(-\beta_k(x_2+(\eta_C)_kx_3+(\eta_A)_kx_4)-\mu_k)\bar{p}_1
+(\beta_k(x_2+(\eta_C)_kx_3+(\eta_A)_kx_4)-\mu_k)\bar{p}_2\\
\quad -\delta_k(x_2+(\eta_C)_kx_3+(\eta_A)_kx_4-\mu_k)\bar{q}_1
+\delta_k(x_2+(\eta_C)_kx_3+(\eta_A)_kx_4)\bar{q}_2+\bar{f}_1]dt
+\bar{q}_1 dB,\\
d\bar{p}_2(t)=-[-\beta_k x_1 \bar{p}_1+(\beta_k x_1-\epsilon_3
-\frac{m_ku}{(1+\delta_k x_2)^2})\bar{p}_2\\
\quad +\left(\phi+\frac{m_ku}{(1+\gamma_k x_2)^2}\bar{p}_3
+e\bar{p}_4-\delta x_1\right)\bar{q}_1
+\delta_k x_1\bar{q}_2+\bar{f}_2]dt+\bar{q}_2 dB,\\
d\bar{p}_3(t) 
=-[-\beta_k (\eta_C)_k x_1 \bar{p}_1+\beta_k (\eta_C)_k x_1 \bar{p}_2
-\epsilon_2 \bar{p}_3-\delta_k (\eta_C)_k x_1 \bar{q}_1
+\delta_k \eta_C x_1 \bar{q}_2+\bar{f}_3]dt,\\
d\bar{p}_4(t) 
=-[-\beta_k (\eta_A)_k x_1 \bar{p}_1
+\beta_k (\eta_A)_k x_1 \bar{p}_2-(\epsilon_1)_k\bar{p}_4
-\delta_k (\eta_A)_k x_1 \bar{q}_1
+\delta_k (\eta_A)_k x_1 \bar{q}_2+\bar{f}_4]dt,
\end{cases}
\end{equation*}
where
\begin{equation*}
\left\lbrace
\begin{aligned}
\bar{f}_1 &=\beta_k[(x_2+(\eta_C)_kx_3+(\eta_A)_kx_4)
-(x'_2+(\eta_C)_kx'_3+(\eta_A)_kx'_4)](p'_2-p'_1)\\
&-(\sigma)_k\left[(x_2+(\eta_C)_kx_3+(\eta_A)_kx_4)
-(x'_2+(\eta_C)_kx'_3+(\eta_A)_kx'_4)\right](q'_2-q'_1)\\
&-L_{x_1}(x,u)+L_{x_1}(x',u'),\\
\bar{f}_2 &= (\beta)_k(x_1-x'_1)(p'_2-p'_1)
+\left[\frac{m_ku}{1+(\eta)_k x_2}
-\frac{(m_ku')_k}{1+(\eta)_k x'_2}\right](p'_3-p'_2)\\
&+(\delta)_k(x_1-x'_1)(q'_2-q'_1)-L_{x_2}(x,u)+L_{x_2}(x',u'),\\
\bar{f}_3 &= (\beta)_k (\eta_C)_k(x_1-x'_1)(p'_2-p'_1)
+ (\delta)_k (\eta_C)_k(x_1-x'_1)(q'_2-q'_1)-L_{x_3}(x,u)+L_{x_3}(x',u'),\\
\bar{f}_4 &= (\beta)_k (\eta_A)_k(x_1-x'_1)(p'_2-p'_1)
+ (\delta)_k (\eta_A)_k(x_1-x'_1)(q'_2-q'_1)-L_{x_4}(x,u)+L_{x_4}(x',u').
\end{aligned}
\right.
\end{equation*}
Let $\phi(t)=\left(\phi_1(t),\phi_2(t),\phi_3(t),\phi_4(t)\right)$
be the solution of the following linear stochastic differential equation: 
\begin{equation}
\label{17}
\left\lbrace
\begin{aligned}
d\phi_1(t) 
&=\Bigl[-(\beta)_k(x_2+(\eta_C)_k x_3+(\eta_A)_k x_4)\phi_1
-(\beta)_k x_1 \phi_2-(\beta)_k (\eta_C)_k 
x_1 \phi_3-\beta_k  (\eta_A)_k x_1 \phi_4\\
&+\mid \bar{p}_1\mid^{\eta-1} \sgn( \bar{p}_1)\Bigr]dt 
+[-(\delta)_k(x_2+(\eta_C)_k x_3+(\eta_A)_k x_4)\phi_1\\
& -(\delta)_k x_1 \phi_2-(\sigma)_k (\eta_C)_k x_1 \phi_3-(\delta)_k
(\eta_A)_k x_1 \phi_4+\mid \bar{q}_1\mid^{\eta-1} \sgn( \bar{q}_1)]dB,\\
d\phi_2(t) &=[(\beta)_k(x_2+(\eta_C)_k x_3+(\eta_A)_k x_4-(\mu)_k)\phi_1
+((\beta)_k x_1-(\epsilon_3)_k-\frac{m_ku}{(1+(\gamma)_k x_2)^2}\phi_2\\
&+(\beta)_k (\eta_C)_k x_1 \phi_3
+\beta_k (\eta_A)_k x_1 \phi_4+\mid \bar{p}_2\mid^{\eta-1} 
\sgn( \bar{p}_2)]dt+[(\delta)_k(x_2+(\eta_C)_k x_3\\
&+(\eta_A)_k x_4)\phi_1+(\delta)_k x_1\phi_2+(\delta)_k (\eta_C)_k
x_1\phi_3+(\delta)_k (\eta_A)_k x_1\phi_4
+\mid \bar{q}_2\mid^{\eta-1} \sgn( \bar{q}_2)]dB,\\
d\phi_3(t) &=[\phi_k+\frac{m_ku}{(1+(\gamma)_k x_2)^2}
-(\epsilon_2)_k\phi_3+\mid \bar{p}_3\mid^{\eta-1} \sgn( \bar{p}_3)]dt,\\
d\phi_4(t) &=[e_k\phi_2-(\epsilon)_k\phi_4+\mid \bar{p}_4\mid^{\eta-1} 
\sgn( \bar{p}_4)]dt.
\end{aligned}
\right.
\end{equation}
Lemma~\ref{3.4} and hypothesis $(S_1)$ show the existence 
and uniqueness of solution to system \eqref{17}.
Using the Cauchy-Schwartz inequality, we obtain the following statement:
$$ 
\sum_{i=1}^4 E\underset{0 \leq t \leq T}{\sup}
\mid \phi_i(t)\mid^{\eta_1}\leq \sum_{i=1}^4 E\int_0^T
\mid \bar{p}_i\mid^\eta dt+\sum_{i=1}^2 E\int_0^T\mid \bar{q}_i\mid^\eta dt,
$$
where $\eta_1>2$ and $\frac{1}{\eta_1}+\frac{1}{\eta}=1$. Set 
$V(\bar{p},\phi)= \displaystyle \sum_{i=1}^4\bar{p}_i\phi_i(t)$.
Using It\^{o}'s formula, we have
\begin{equation*}
\begin{aligned}
\sum_{i=1}^4 E & \int_0^T \mid \bar{p}_i\mid^\eta dt
+\sum_{i=1}^2 E\int_0^T \mid \bar{q}_i\mid^\eta dt\\
&= -\sum_{i=1}^4 E\int_0^T\phi_i\bar{f}_i dt
+\sum_{i=1}^4E\phi_i(T)\left[ h_{x_i}(x(T))-h_{x_i}(x'(T))\right]\\
&\leq R \sum_{i=1}^4 \left( E\int_0^T \mid \bar{f}_i\mid^\eta dt 
\right) ^{\frac{1}{\eta}}\left( E\int_0^T 
\mid \phi_i\mid^{\eta_1}dt\right) ^{\frac{1}{\eta_1}}\\
& \quad + R \sum_{i=1}^4 \left( E \left[  h_{x_i}(x(T))-h_{x_i}(x'(T))
\right]^{\eta} \right)^{\frac{1}{\eta}}\times \left( E \mid \phi_i(T) 
\mid^{\eta_1} \right) ^{\frac{1}{\eta_1}} \\
& \leq R \left(  \sum_{i=1}^4 E\int_0^T \mid \bar{p}_i\mid^\eta dt
+\sum_{i=1}^2 E\int_0^T \mid \bar{q}_i\mid^{\eta} dt \right)^{\frac{1}{\eta_1}}\\
& \quad \times \left(\sum_{i=1}^4 \left( E \int_0^T \mid \bar{f}_i
\mid^{\eta} dt \right) ^{\frac{1}{\eta}}+\sum_{i=1}^4\left( 
E \mid h_{x_i}(x(T)-h_{x_i}(x'(T))\mid ^\eta)^{\frac{1}\eta}\right) \right).
\end{aligned}
\end{equation*}
By the elementary inequality \eqref{el:ineq}, we have
\begin{multline*}
\sum_{i=1}^4 E\int_0^T \mid \bar{p}_i\mid^\eta dt
+\sum_{i=1}^2 E\int_0^T \mid \bar{q}_i\mid^\eta dt\\
\leq R\left[ \sum_{i=1}^4E \int_0^T \mid \bar{f}_i\mid^\eta dt
+\sum_{i=1}^4E\mid h_{x_i}(x(T))-h_{x_i}(x'(T))\mid^\eta\right].
\end{multline*}
Using $(S_2)$ and Lemma~\ref{3.3}, we obtain that
$$
\sum_{i=1}^4E\mid h_{x_i}(x(T))-h_{x_i}(x'(T))\mid^\eta
\leq R^\eta\sum_{i=1}^4E\mid x_i(t)-x'_i(t)\mid^\eta]
\leq R d(u,u')^{\frac{k\eta}{2}}.
$$
It follows from the Cauchy--Schwartz inequality that
\begin{equation*}
\begin{aligned}
E\int_0^T \mid \bar{f}_i\mid^\eta dt
& \leq R \Biggl[ E\int_0^T\mid x_2-x'_2\mid^\eta\left( 
\sum_{i=1}^4\mid p'_i\mid^\eta+\sum_{i=1}^2\mid q'_i\mid^\eta\right) dt \\
& \qquad\quad + E\int_0^T\mid L_{x_1}(x,u)- L_{x_1}(x',u')\mid ^\eta dt\Biggr] \\
& \leq R\left(  E\int_0^T\mid x_2-x'_2\mid ^{\frac{2\eta}{2-\eta}}dt
\right)^{1-\frac{\eta}{2}}\\
&\qquad\quad \times \left[ \left( \sum_{i=1}^2
\int_0^T \mid p'_i\mid^2 dt\right) ^{\frac{\eta}{2}}
+\left( \sum_{i=1}^2\int_0^T 
\mid q'_i\mid^2 dt\right)^{\frac{\eta}{2}}\right]  
+ R d(u,u')^{\frac{k\eta}{2}}.
\end{aligned}
\end{equation*}
Observe that $\frac{2\eta}{1-\eta}<1$, 
$1-\frac{\eta}{2}>\frac{k\eta}{2}$, and $d(u,u')<1$.
It comes from \eqref{3.4} and \eqref{3.5} that
$$
E\left[\int_0^T \mid \bar{f}_1\mid^\eta dt\right]
\leq d(u,u')^{\frac{k\eta}{2}}.
$$
Similarly, and omitting the details, 
$$
\sum_{i=1}^4E\int_0^T \mid \bar{f}_i\mid^\eta dt
\leq R d(u,u')^{\frac{k\eta}{2}}.
$$
It results that
$$
\sum_{i=1}^4E\int_0^T \mid \bar{p}_i\mid^\eta dt
+\sum_{i=1}^2E\int_0^T \mid 
\bar{q}_i\mid^\eta dt \leq R d(u,u')^{\frac{k\eta}{2}}
$$
and the proof is complete.
\end{proof}


\subsection{Necessary condition for near-optimal control}
\label{sec3.3}

Our necessary condition for the near-optimal control 
of system \eqref{13} makes use of the 
following classical result.

\begin{lemma}[Ekeland's variational principle \cite{R6}]
\label{3.1}
Let $(Q,d)$ be a complete metric space and $F :Q \rightarrow \mathbb{R}$ 
be a lower-semi continuous function bounded from below. For any $\epsilon>0$, 
assume that $u^\epsilon\in Q$ satisfies
\begin{equation}
\label{12}
F(u^\epsilon)\leq \underset {u \in Q}{\inf}F(u)+\epsilon.
\end{equation}
Then, there exists a $u^\lambda\in Q$, $\lambda>0$, 
such that for all $u\in Q$ one has
$$
F(u^\lambda)\leq F(u^\epsilon),
\quad d(u^\lambda,u^\epsilon) \leq \lambda,
\quad F(u^\lambda) \leq F(u)
+\frac{\epsilon}{\lambda}d(u^\lambda,u^\epsilon).
$$
\end{lemma}

\begin{theorem}
\label{Th2}
Let $(S_1)$ and $(S_2)$ hold, $x \mapsto L(x,u)$ and $x \mapsto h(x)$ 
be convex almost surely, and $(p^\epsilon(t)$, $q^\epsilon(t))$ 
be the solution of the adjoint equation under control $u^\epsilon$. 
Then there exists an imprecise parameter $R$
such that for any $\eta\in[0,1]$, $\epsilon>0$ and any $\epsilon$-optimal 
pair $\left(x^\epsilon(t),u^\epsilon(t)\right)$, the following condition holds:
\begin{multline*}
\underset{u\in U_{ad}}{\inf}E\left\{\int_0^T \left(
\frac{m_ku(t)I^\epsilon(t)}{1+\gamma_k I^\epsilon(t)}(p_3^\epsilon(t)
-p_2^\epsilon(t))+L(x^\epsilon(t),u(t))\right)dt\right\}\\
\geq E\left\{\int_0^T \left(\frac{m_ku^\epsilon(t)I^\epsilon(t)}{1+\gamma_k
I^\epsilon(t)}(p_3^\epsilon(t)-p_2^\epsilon(t))
+L(x^\epsilon(t),u^\epsilon(t))\right)dt\right\}
-R\epsilon^{\frac{\eta}{3}}.
\end{multline*}
\end{theorem}

\begin{proof}
The function $J(x_0,\cdot):U_{ad}\rightarrow\mathbb{R}$ 
is continuous under the metric $d$ defined by \eqref{eq:metric:d}. 
Applying Lemma~\ref{3.1}, and taking $\lambda=\epsilon^{\frac{2}{3}}$, there exists an 
admissible pair $\left(\bar{x}^\epsilon(t),\bar{u}^\epsilon(t)\right)$ such that
\begin{equation}
d(u^\epsilon(t),\bar{u}^\epsilon(t))\leq \epsilon^{\frac{2}{3}},
\end{equation}
\begin{equation}
\label{21}
J\left(0,x_0,\bar{u}^\epsilon(t)\right)
\leq \bar{J}(0,x_0,u)
\end{equation}
for all $u\in U_{ad}[0,T]$, where
\begin{equation}
\label{22}
\bar{J}(0,x_0,u)= J(0,x_0,u)+\epsilon^{\frac{1}{3}} 
d\left(u^\epsilon,\bar{u}^\epsilon\right).
\end{equation}
This is equivalent to the fact that $\left(\bar{x}^\epsilon(t),\bar{u}^\epsilon(t)\right)$ 
is an $\epsilon$-optimal pair for the system \eqref{13} under the cost functional \eqref{6}.
Moreover, a necessary condition for $\left(\bar{x}^\epsilon(t),\bar{u}^\epsilon(t)\right)$ 
is deduced by the following setting. Let $\bar{t} \in [0,T]$, 
$\rho > 0$ and $u \in U_{ad}[0,T]$. We define
$u^\rho(t)=u(t)$ if $ t\in[\bar{t},\bar{t}+\rho]$, and 
$u^\rho(t)=\bar{u}^\epsilon(t)$ if 
$t\in [0,T] \backslash \left[\bar{t},\bar{t}+\rho\right]$.
We deduce from equations \eqref{21} and \eqref{22} that
\begin{equation}
\label{23}
\bar{J}(0,x_0,\bar{u}^\epsilon) \leq \bar{J}(0,x_0,u^\rho)
\end{equation}
and
\begin{equation}
\label{24}
d(u^\rho,\bar{u}^\epsilon)\leq \epsilon^{\frac{2}{3}}.
\end{equation}
It comes from \eqref{23}, \eqref{24} and Taylor's expansion that
\begin{equation}
\begin{aligned}
-\rho\epsilon^{\frac{1}{3}}
&\leq J(0,x_0,u^\rho)-J(0,x_0,\bar{u}^\epsilon)\\
&= E \int_0^T \left[ L(x^\rho,u^\rho)-L(\bar{x}^\epsilon,
\bar{u}^\epsilon) \right]dt+E[h(x^\rho(T))-h(\bar{x}^\epsilon(T))]\\
&\leq \sum_{i=1}^4E\left\{\int_0^TL_{x_i}(\bar{x}^\epsilon,
u^\rho)\left(x_i^\rho-\bar{x}_i^\epsilon\right)dt\right\}\\
&\quad +E\left\{\int_{\bar{t}}^{\bar{t}+\rho}[L(\bar{x}^\epsilon,u)
-L(\bar{x}^\epsilon,\bar{u}^\epsilon)]dt\right\}\\
&\quad +\sum_{i=1}^4E\left[h_{x_i}(\bar{x}^\epsilon(T))(x_i^\rho(T)
-\bar{x}_i^\epsilon(t))\right]+o(\rho).
\end{aligned}
\end{equation}
The It\^{o} formula applied on 
$\sum_{i=1}^4 \bar{p}_i^\epsilon(x_i^\rho- \bar{x}_i^\epsilon)$ 
and the use of  Lemmas~\ref{3.2} and \ref{3.4} yield
$$ 
\sum_{i=1}^4 E\left[h_{x_i}(x^\rho(T))(x_i^\rho(T)
-\bar{x}_i^\epsilon(T))\right]
\leq E\left\{\int_{\bar{t}}^{\bar{t}+\rho} 
\left[(u^\rho-\bar{u}^\epsilon)(\bar{p}_3^\epsilon)
-\bar{p}_1^\epsilon+(u^\rho-\bar{u}^\epsilon)
\bar{p}_3^\epsilon\right]dt\right\}.
$$
Subsequently,
\begin{equation}
\begin{aligned}
-\rho\epsilon^{\frac{1}{3}} 
&\leq J(0,x_0,u^\rho)-J(0,x_0,\bar{u}^\epsilon)\\
&\quad +E\left\{\int_{\bar{t}}^{\bar{t}+\rho}\left[
L(\bar{x}^\rho,u)-L(\bar{x}^\epsilon,\bar{u}^\epsilon)\right]dt\right\}\\
&\quad +E\left\{\int_{\bar{t}}^{\bar{t}+\rho}[(u(t)-\bar{u}^\epsilon)(t)]
\left[\left(\bar{p}_3(t)-\bar{p}_1(t)\right)
+\left(u(t)-\bar{u}^\epsilon\right)
\bar{p}_3^\epsilon\right]dt\right\}+o(\rho).
\end{aligned}
\end{equation}
Dividing by $\rho$ and letting $\rho\rightarrow 0$, we have
\begin{multline}
-\epsilon^{\frac{1}{3}}
\leq E\left[L(\bar{x}^\epsilon(t),u(\bar{t}))
-L(\bar{x}^\epsilon(t),\bar{u}^\epsilon(t))\right]\\
+E\left[\left(u(\bar{t})-\bar{u}^\epsilon(t)\right)
\left(\left(\bar{p}_3(t)-\bar{p}_1(t)\right)
+\left(u(t)-\bar{u}^\epsilon\right)\right)
\bar{p}_3^\epsilon(t)\right].
\end{multline}
We estimate the following variation:
\begin{equation*}
\begin{aligned}
E&\left\{\int_0^T\left[(u^\rho(t)-\bar{u}^\epsilon(t))
\bar{p}_3^\epsilon(t)-\left(u^\rho(t)-u^\epsilon(t)\right)
p_3^\epsilon(t)\right]dt\right\}\\
&=E\left\{\int_0^T\left(\bar{p}_3^\epsilon(t)-p_3^\epsilon\right)
\left(u^\rho(t)-u^\epsilon(t)\right)dt\right\}
+\left\{\int_0^T \left(p_3^\epsilon(t)\left(u^\epsilon(t)
-\bar{u}^\epsilon\right)(t)\right)dt\right\}\\
& =I_1+I_2.
\end{aligned}
\end{equation*}
Using Lemma~\ref{3.5}, we conclude that for $0<k<1$ and $1<\eta<2$ 
verifying $(1+k)\eta<2$ one has
\begin{equation*}
\begin{aligned}
I_1 & \leq \left(E\int_0^T \mid \bar{p}_3^\epsilon(t)
-p_3^\epsilon(t)\mid ^\eta dt\right)^{\frac{1}{\eta}}\left(
E\int_0^T \mid u^\rho(t)-u^\epsilon(t)\mid^{
\frac{\eta}{\eta-1}} dt\right)^{\frac{\eta - 1}{\eta}}\\
& \leq R\left(d\left(u^\rho,u^\epsilon\right)^{\frac{k\eta}{2}}\right)^{\frac{1}{\eta}}
\left(E\int_0^T \left(\mid u^\rho(t)\mid ^{\frac{\eta}{\eta-1}}
+\mid u^\epsilon(t)\mid^{\frac{\eta}{\eta-1}}\right)
dt\right)^{\frac{\eta-1}{\eta}}\\
& \leq R\epsilon^{\frac{k}{3}}
\end{aligned}
\end{equation*}
and
\begin{equation*}
\begin{aligned}
I_2 
& \leq R\left(E\int_0^T \mid p_3^\epsilon(t)\mid ^2 dt\right)^{\frac{1}{2}}
\left(E\int_0^T \mid u^\epsilon(t)-\bar{u}^\epsilon(t)\mid^2
\chi_{u^\epsilon(t)\neq \bar{u}^\epsilon(t)}(t)dt\right)^{\frac{1}{2}}\\
& \leq R\left(E\int_0^T \left(\mid u^\epsilon(t)\mid^4
+\mid \bar{u}^\epsilon(t)\mid ^4\right)dt\right)^{\frac{1}{4}}
\left(E\int_0^T\chi_{u^\epsilon(t)
\neq \bar{u}^\epsilon(t)}(t)dt\right)^{\frac{1}{4}}.
\end{aligned}
\end{equation*}
Thus,
\begin{equation}
E\int_0^T\left[\left(u^\rho(t)-\bar{u}^\epsilon(t)\right)
\left(\bar{p}_3^\epsilon(t)-\left(\bar{p}_1^\epsilon(t)\right)
-\left(u^\rho(t)-u^\epsilon(t)\right)
p_3^{\epsilon}(t)\right)\right]dt
\leq \epsilon^{\frac{k}{3}}.
\end{equation}
With a similar argument, we obtain that
\begin{equation}
\begin{aligned}
E&\left\{\int_0^T\left[\left(u^\rho(t)-\bar{u}^\epsilon(t)\right)
\left(\bar{p}_3^\epsilon(t)\bar{p}_1^\epsilon(t)\right)
-u^\rho(t)-u^\epsilon(t)
\left(p_3^\epsilon(t)-p_1^\epsilon(t)\right)\right]dt\right\}\\
& +E\left\{\int_0^T[L(\bar{x}^\epsilon(t),u^\rho(t))
-L(\bar{x}^\epsilon(t),\bar{u}^\epsilon(t))]
-\left[L\left(x^\epsilon(t),u^\rho(t)\right)
-L\left(x^\epsilon(t),u^\epsilon(t)\right)\right]dt\right\}\\
& \leq \epsilon^{\frac{k}{3}}.
\end{aligned}
\end{equation}
We obtain the desired result from the above inequalities. 
\end{proof}


\subsection{Sufficient condition for near-optimal control}
\label{sec3.4}

Besides $(S_1)$ and $(S_2)$, now we also impose a further hypothesis:
\begin{enumerate}
\item[$(S_3)$] The set $U$ where the control takes values is convex.
\end{enumerate}

\begin{theorem}
\label{Th3}
Suppose that hypothesis $(S_1)$, $(S_2)$ and $(S_3)$ hold. 
Let $\left(x^\epsilon(t),u^\epsilon(t)\right)$ be an admissible pair 
and $\left(p^\epsilon(t),q^\epsilon(t)\right)$ be the solution of the 
adjoint equation corresponding to $\left(x^\epsilon(t),u^\epsilon(t)\right)$.
Assume that $(x,u) \mapsto H(t,x,u,p,q)$ and $x \mapsto h(x)$ are convex almost surely. 
If for some $\epsilon>0$,
\begin{multline*}
E\int_0^T\left(\frac{m_ku^\epsilon(t)x_2^\epsilon(t)}{1
+\gamma_kx_2^\epsilon(t)}\left(p_3^\epsilon-p_2^\epsilon\right)
+L\left(x^\epsilon(t),u^\epsilon(t)\right)\right)dt\\
\leq \underset{u\in U_{ad}}{\inf}
E\int_0^T \frac{m_ku(t)x_2^\epsilon(t)}{1
+\gamma_kx_2^\epsilon(t)}\left(p_3^\epsilon-p_2^\epsilon\right)
+L\left(x^\epsilon(t),u(t)\right)dt+\epsilon,
\end{multline*}
then
$J(0,x_0,u^\epsilon)
\leq \underset{u \in U_{ad}}{\inf}
J(0,x_0,u)+ R\epsilon^{\frac{1}{3}}$.
\end{theorem}

\begin{proof}
From the definition of the Hamiltonian function \eqref{9}, we have
\begin{equation}
\label{30}
\begin{aligned}
J(0,x_0,u^\epsilon)-J(0,x_0,u)
=I_1(t)+I_2(t)+I_3(t)
\end{aligned}
\end{equation}
with
\begin{equation}
\begin{aligned}
I_1= & E\int_0^T[H(t,x^\epsilon(t),u^\epsilon(t),p^\epsilon(t),q^\epsilon(t))
-H(t,x(t),u(t),p^\epsilon(t),q^\epsilon(t))]dt,\\
I_2= & E[h(x^\epsilon(T))-h(x(T))],\\
I_3= & E\int_0^T\Biggl[(f^\top(x^\epsilon(t),u^\epsilon(t))
-f^\top(x(t),u(t)))p^\epsilon(t)\\
&\qquad +(\sigma^\top(x^\epsilon(t),u^\epsilon(t))
-\sigma^\top(x(t),u(t)))q^\epsilon(t)\Biggr]dt.
\end{aligned}
\end{equation}
Using the convexity of $H$, we obtain that
\begin{multline}
\label{32}
I_1 \leq \sum_{i=1}^4E
\int_0^T H_{x_i}\left(t,x^\epsilon(t),u^\epsilon(t),
p^\epsilon(t),q^\epsilon(t)\right)
\left(x_i^\epsilon(t)-x_i(t)\right)dt\\
+ E\int_0^T 
H_u\left(t,x(t),u(t),p^\epsilon(t),q^\epsilon(t)\right)
\left(u^\epsilon(t)-u(t)\right)dt.
\end{multline}
Similarly,
\begin{equation}
\label{33}
I_2 \leq \sum_{i=1}^4E(h_{x_i}(x^\epsilon(T)-x_i(T))).
\end{equation}
The It\^{o}'s formula acting on $\sum_{i=1}^4p_i^\epsilon(t)(x_i^\epsilon(t)-x_i(t))$ 
yields
\begin{equation*}
\begin{aligned}
\sum_{i=1}^4E\left(h_{x_i}\left(x^\epsilon(T)-x_i(T)\right)\right)
&= -\sum_{i=1}^4E \int_0^T\left(x_i^\epsilon(t)-x_i(t)\right)
H_{x_i}\left(t,x^\epsilon(t),u^\epsilon(t),
p^\epsilon(t),q^\epsilon(t)\right)dt\\
& \quad +\sum_{i=1}^4 E\int_0^T p_i^\epsilon(t)\mid 
f_i\left(x^\epsilon(t),u^\epsilon(t)\right)
-f_i\left(x(t),u(t)\right)\mid dt\\
&\quad +\sum_{i=1}^4 E \int_0^T
q_i^\epsilon(t)\mid\sigma_i(x^\epsilon(t))
-\sigma(x(t))\mid dt.
\end{aligned}
\end{equation*}
Hence,
\begin{multline}
\label{34}
I_3=\sum_{i=1}^4 
E\left(h_{x_i}(x^\epsilon(T))(x_i^\epsilon(T)-x_i(T))\right)\\
+\sum_{i=1}^4 E \int_0^T \left(x_i^\epsilon(t)-x_i(t)\right)
\cdot H_{x_i}\left(t,x^\epsilon(t),u^\epsilon(t),
p^\epsilon(t),q^\epsilon(t)\right)dt.
\end{multline}
Replacing equations \eqref{32}, \eqref{33} and \eqref{34} 
into \eqref{30}, we obtain that
\begin{equation}
\label{eq:rhs:int:Hu}
J(0,x_0,u^\epsilon)-J(0,x_0,u) 
\leq E\int_0^T H_u\left(t,x^\epsilon(t),u^\epsilon(t),
p^\epsilon(t),q^\epsilon(t)\right)
\left(u^\epsilon(t)-u(t)\right) dt.
\end{equation}
To finish the proof, we need to estimate the right-hand side of \eqref{eq:rhs:int:Hu}.
Consider the metric $\bar{d}$ on $U_{ad}$ defined by
\begin{equation}
\label{eq:3.27}
\bar{d}(u,u')=E\int_0^Ty^\epsilon(t)\mid u(t)-u'(t)\mid dt,
\end{equation}
where
\begin{equation}
\label{eq:3.28}
y^\epsilon(t)=1+\sum_{i=1}^4\mid p_i^\epsilon(t)\mid
+\sum_{i=1}^2\mid q_i^\epsilon(t)\mid.
\end{equation}
It is straightforward to state that $\bar{d}$ 
is a complete metric as a weighted $L^1$ norm.
Let functional 
$F(\cdot) : U_{ad}\rightarrow\mathbb{R}$ be defined by
\begin{equation}
\label{eq:3.30}
F(u^\epsilon)= E\int_0^T
H\left(t,x^\epsilon(t),u^\epsilon(t),p^\epsilon(t),q^\epsilon(t)\right)dt.
\end{equation}
Using $(S_2)$, we see that $F(\cdot)$ is continuous with respect to the metric $\bar{d}$.
Thus, from Lemma~\ref{3.1}, there exists a $\bar{u}^\epsilon \in U_{ad}$ such that
\begin{equation}
\label{eq:3.31}
\begin{aligned}
& \bar{d}(u^\epsilon,\bar{u}^\epsilon)
\leq \epsilon^{\frac{1}{2}},\\
& F(\bar{u}^\epsilon)\leq F(u)+\epsilon^{\frac{1}{2}} 
\bar{d}(u^\epsilon,\bar{u}^\epsilon),
\end{aligned}
\end{equation}
for any $u\in U_{ad}$. Replacing 
\eqref{eq:3.30} into \eqref{eq:3.31}, 
it follows from the differentiability 
of function $H$ with respect to $u$ 
and hypothesis $(S_1)$ that
\begin{equation}
\label{44}
H_u\left(t,x^\epsilon(t),u^\epsilon(t),p^\epsilon(t),
q^\epsilon(t)\right)+\lambda_1^\epsilon(t)=0
\end{equation}
with 
\begin{equation}
\lambda_1^\epsilon(t)\in\left[-\epsilon^{\frac{1}{2}}
y^\epsilon(t),\epsilon^{\frac{1}{2}}y^\epsilon(t)\right].
\end{equation}
We conclude that
\begin{equation}
\label{eq:3.35}
\begin{aligned}
\mid H_u\left(t,x^\epsilon(t),u^\epsilon(t), 
p^\epsilon(t),q^\epsilon(t)\right)\mid 
&=  \mid\lambda_1^\epsilon(t)\mid\\
& \leq  2 \epsilon^{\frac{1}{2}} \left|y^\epsilon(t)\right|.
\end{aligned}
\end{equation}
The result follows from \eqref{eq:rhs:int:Hu}, \eqref{eq:3.28},
\eqref{eq:3.35}, Lemma~\ref{3.4}, and H\"{o}lder's inequality.
\end{proof}


\section{Conclusion}
\label{sec4}

Parameter values are usually considered to be precisely known 
in epidemic mathematical modeling. However, often they are imprecise 
due to various uncertainties. Therefore, here we have proposed
the near-optimal control of a stochastic epidemic 
SICA model with imprecise parameters. By using some mathematical
inequalities, namely Cauchy-Schwartz's, Gronwall's and
Burkholder--Davis--Gundy inequalities, we proved some estimates 
of the state and co-state variables in order to investigate
necessary and sufficient conditions of near-optimality. 

As future work, we plan to develop numerical methods
for near-optimal control of the stochastic SICA model
with imprecise parameters. This is under investigation 
and will be addressed elsewhere.


\section*{Acknowledgments}

This work is part of first author's PhD project,
which is carried out at University of Aveiro. The authors 
were supported by the Portuguese Foundation for Science 
and Technology (FCT) under Grant No. UIDB/04106/2020 (CIDMA).
Valuable comments and suggestions of improvement, 
from Reviewer and Editor, are here gratefully acknowledged.



\end{document}